\documentclass{amsart}

\usepackage{amssymb,amsfonts,amscd,amstext,amsmath,amsthm}
\usepackage{listings,color,enumerate,multirow,graphicx}
\usepackage{float,setspace,MnSymbol}
\usepackage{verbatim} 
\usepackage{tikz,scalerel,relsize,xfrac}
\usepackage[bookmarksopen=true,bookmarksnumbered=true]{hyperref} 
\usepackage{letltxmacro}
\usepackage{multicol}
\usepackage{url} 
\usepackage{bm} 
\theoremstyle{plain}
\allowdisplaybreaks

\theoremstyle{plain}
\newtheorem{theorem}{Theorem}[section]
\newtheorem{proposition}[theorem]{Proposition}
\newtheorem{lemma}[theorem]{Lemma}
\newtheorem{corollary}[theorem]{Corollary}

\theoremstyle{definition}
\newtheorem{definition}[theorem]{Definition}

\newtheorem{example}[theorem]{Example}

\theoremstyle{remark}
\newtheorem{remark}[theorem]{Remark}

\let\olddownarrow\downarrow 
\renewcommand{\downarrow}{\mathord{\olddownarrow}}
\let\olduparrow\uparrow 
\renewcommand{\uparrow}{\mathord{\olduparrow}}
\def\e#1{\emph{#1}}

\def\i#1{\textit{#1}}
\def\b#1{\textbf{#1}}
\def\t#1{\text{#1}}
\def\mc#1{\mathcal{#1}}
\def\sub{\subseteq}

\def\dua{\twoheaduparrow}
\def\dda{\twoheaddownarrow}

\def\duam{\twoheaduparrow_{\mathcal{M}}}
\def\ddam{\twoheaddownarrow_{\mathcal{M}}}
\def\duamn{\twoheaduparrow_{\mathcal{MN}}}
\def\ddamn{\twoheaddownarrow_{\mathcal{MN}}}
\def\duatr{\rotatebox[origin=c]{90}{$-\!\!\!-\!\!\!\!\!\smalltriangleright$}}
\def\ddatr{\rotatebox[origin=c]{270}{$-\!\!\!-\!\!\!\!\!\smalltriangleright$}}
\def\duatrmn{\duatr_{\mc{MN}}}
\def\ddatrmn{\ddatr_{\mc{MN}}}

\def\llm{\ll_{\mathcal{M}}}
\def\llmn{\ll_{\mathcal{MN}}}
\def\trmn{\triangleleft_\mathcal{MN}}

\def\mconv{\xrightarrow{\mathcal{M}}}
\def\mnconv{\xrightarrow{\mathcal{MN}}}

\newcommand{\Dir}{\operatorname{Dir}}
\newcommand{\Filt}{\operatorname{Filt}}
\newcommand{\Irr}{\operatorname{Irr}}
\newcommand{\Red}{\operatorname{Red}}

\newcommand{\ub}{\operatorname{ub}}
\newcommand{\lb}{\operatorname{lb}}

\newcommand{\op}{\operatorname{op}}

\def\himp#1{\{#1\}} 
\newcommand{\brackets}[1]{\left(#1\right)}
\newcommand{\sets}[1]{\left\{#1\right\}}
\newcommand{\elb}{\operatorname{elb}}
\def\elbb#1{\elb\brackets{#1}}
\usepackage[utf8]{inputenc}
\usepackage{geometry}
\usepackage{graphicx}

\begin{document}

\title[Generalised Net Convergence Structures]{Generalised Net Convergence Structures in Posets}

\author[H. Andradi]{Hadrian Andradi}
\address{\noindent Hadrian Andradi \newline\indent National Institute of Education\newline\indent Nanyang Technological University\newline\indent 1 Nanyang Walk, Singapore 637616,\newline\indent Department of Mathematics\newline\indent Faculty of Mathematics and Natural Sciences\newline\indent Universitas Gadjah Mada, Indonesia 55281}
\email{hadrian.andradi@gmail.com}
\thanks{The first author is supported by Nanyang Technological University Research Scholarship}

\author[W.K. Ho]{Weng Kin Ho}
\address{Weng Kin Ho\newline\indent National Institute of Education\newline\indent Nanyang Technological University\newline\indent 1 Nanyang Walk, Singapore 637616}
\email{wengkin.ho@nie.edu.sg}

\maketitle

\begin{abstract}
In this paper, we introduce the notion of $\mc{M}$-convergence and $\mc{MN}$-convergence structures in posets, which, in some sense, generalise the well-known Scott-convergence and order-convergence structures. As results, we give a necessary and sufficient conditions for each generalised convergence structures being topological. These results then imply the following two well-established results: (1) The Scott-convergence structure in a poset $P$ is topological if and only if $P$ is continuous, and (2) The order-convergence structure in a poset $P$ is topological if and only if $P$ is $\mc{R}^*$-doubly continuous.
\end{abstract}

\section{Introduction}
In the study of posets, there is a well-known convergence structures in posets, namely Scott-convergence structure, which was initially defined by D. Scott in complete lattice. This convergence structure is then generalised to the realm of general posets in \cite{zhaozhao05}. Recall that a net $\brackets{x_i}_{i\in I}$ in  a poset $P$ is said to \e{Scott-converge} (or \e{$\mc{S}$-converge}) \e{to} $x\in P$, denoted by $\brackets{x_i}_{i\in I}\xrightarrow{\mc{S}}x$, if there exists a directed subset $D$ of $P$ such that 
\begin{enumerate}
\item $x\leq \bigvee D$, and
\item for each $d\in D$, there exists $i_0\in I$ such that $i\geq i_0$ implies $d\leq x_i$.
\end{enumerate}

It is known that every convergence structure $\mc{C}$ in a given set induces a topology $\tau_{\mc{C}}$. Conversely, every topology $\tau$ generates a convergence structure $\mc{C}_{\tau}$ in the underlying set. A convergence structure $\mc{C}$ in a set $X$ is said to be \e{topological} if the convergence structure generated by the topology $\tau_{\mc{C}}$ is precisely $\mc{C}$.

In general, the net Scott-convergence structure in a poset $P$ is not topological. There may be a net and a filter which converges topologically to $x\in P$ with respect to $\tau_{\mc{S}}$ but does not $\mc{S}$-converge to $x$. Indeed, the net always exists unless the poset is continuous. Conversely, continuity of poset can guarantee that such net does not exist. The last two facts are summarised in the following theorem.

\begin{theorem}\cite{zhaozhao05}\label{th:1}
The net Scott-convergence structure in a poset $P$ is topological if and only if $P$ is a continuous poset.
\end{theorem}

Another well-studied convergence structure in a poset is order-convergence (see, e.g., \cite{ole99,sunli18,wangzhao13,wolk61,zhaoli06,zhaolu17,zhaowang11,zhouzhao07}). A net $\brackets{x_i}_{i\in I}$ in  a poset $P$ is said to \e{order-converge to $x\in P$}, denoted by $\brackets{x_i}_{i\in I}\xrightarrow{\mc{O}}x$, if there exist a directed subset $D$ and a filtered subset $F$ of $P$ such that 
\begin{enumerate}
\item $x= \bigvee D=\bigwedge F$, and
\item for each $d\in D$ and $e\in F$, there exists $i_0\in I$ such that $i\geq i_0$ implies $d\leq x_i\leq e$.
\end{enumerate}
Just like net Scott-convergence structure in a poset, net order-convergence structure may not be topological. Some attempts  to find a characterisation for posets in which the order-convergence structure being topological were undergone, but no complete characterisation was obtained, until Sun Tao and Qingguo Li established a necessary and sufficient condition for the order-convergence being topological in \cite[Theorem 4.13]{sunli18}.
\begin{theorem}\label{th:2}\cite{sunli18}
The net order-convergence structure in a poset $P$ is topological if and only if $P$ is an $\mc{R}^*$-doubly continuous poset.
\end{theorem}
Of course, order-convergence structure can be written in terms of filter convergence. A similar characterisation considering filter order-convergence structure is also established in \cite{sunliguo16}. We recall that a poset $P$ is \e{doubly continuous} if $P$ and $P^{\op}$ are continuous. In view of Theorem \ref{th:1}, doubly-continuity seems to be a natural characterisation for order convergence being topological. Unfortunately, that is not the case as given in the following.

\begin{theorem}\label{th:3}\cite{zhouzhao07}
Let $P$ satisfy Condition $(\ast)$. Then the net convergence structure in $P$ is topological if and only if $P$ is doubly-continuous.
\end{theorem}

Consider the following definition: A poset $P$ is \e{meet-continuous} if for each $D\subseteq P$ and $x\in D$, $x\leq\bigvee D$ implies $\bigvee\sets{x\wedge d\mid d\in D}$ exists and equals $x$. Then one can simply say that a poset $P$ satisfies Condition $(\ast)$ if and only if $P$ is \e{doubly-meet-continuous}, i.e., both $P$ and $P^{\op}$ are meet-continuous. In \cite[Example 3.6]{sunli18}, an example of an $\mc{R}^*$-doubly continuous poset which fails to be doubly continuous poset is presented. This means that there is a poset not satisfying Condition $(\ast)$ in which the order-convergence structure is topological. Ones may be polarised when seeing Theorem \ref{th:2} and Theorem \ref{th:3}. Some may say that the characterisation given in Theorem \ref{th:3} is not complete since it only holds in a certain class of posets, but it is good because of the naturalness of doubly-continuity. Some others may say that the characterisation given in Theorem \ref{th:2} is complete, but the definition of $\mc{R}^*$-doubly continuity is not really natural. In this paper, we put our focus more on complete characterisations, following Theorem \ref{th:2}.

In the development of domain theory, the idea of generalising existing classical results is of interest. One branch of generalisation is via subset selections (see, e.g., \cite{bar96,ven86,zhangli17,zhao92,zhao15}). Recall that a \e{subset selection} is an assignment $\mc{M}$ assigning every poset $P$ to a collection $\mc{M}(P)$ of subsets of $P$. In \cite{zhouzhao07}, the notion of lim-inf$_{\mc{M}}$-convergence (which we will call $\mc{M}$-convergence in this paper) is studied. Scott-convergence is a special case of this convergence. As expected, this convergence is not always topological. A characterisation of poset in which the lim-inf$_{\mc{M}}$-convergence structure is topological (given $\mc{M}$ contains all singletons as members). This characterisation is considered incomplete in the sense of \cite{sunli18}, i.e., the poset needs to satisfy some certain condition at the first place. In other words, the characterisation only holds in some certain class of posets (see \cite[Theorem 3.1]{zhouzhao07}). In this paper, we provide a complete characterisation by introducing the notion of $\mc{M}$-continuous space. This result will have Theorem \ref{th:1} as a corollary. We also define the notion of $\mc{MN}$-convergence to generalise order-convergence in the same spirit as lim-inf$_{\mc{M}}$-convergence generalises Scott-convergence. We will then provide a characterisation of a poset in which $\mc{MN}$-convergence structure being topological. From this condition, we then can deduce Theorem \ref{th:2} given above.

\section{Preliminaries}

We will only provide some denotations and convergence-related definitions and result in this section. For any other standard definition and notation of topology and domain theory, we refer to \cite{redbook}, \cite{gl}, and \cite{kel}. 

Let $P$ be a poset and $Q\sub P$. We denote the set of all upper bounds, resp. lower bounds, of $Q$ by $\ub(Q)$, resp $\lb(Q)$. Let $X$ be a set and $\brackets{x_i}_{i\in I}$ be a net in $X$. We say that $x_i$ \e{satisfies condition $\alpha$ eventually} if there exists $i_0\in I$ such that for all $i\in I$, $i\geq i_0$ implies $x_i$ satisfies condition $\alpha$. An element $x$ is an \e{eventual lower bound of $\brackets{x_i}_{i\in I}$} if $x_i\geq x$ eventually. We denote the set of all eventual lower bounds of $\brackets{x_i}_{i\in I}$ by $\elbb{\brackets{x_i}_{i\in I}}$. A \emph{net convergence structure} in $X$ is a class $\mathcal{C}$ of tuples $\brackets{(x_i)_{i\in I},x}$ where $(x_i)_{i\in I}$ is a net whose terms are elements of $X$ and $x\in X$. A topology induced by a net convergence structure $\mathcal{C}$ in $X$ is the collection $\tau_{\mathcal{C}}$ of all subsets $U$ of $X$  satisfying
$$((x_i)_{i\in},x)\in\mathcal{C} \text{ and }x\in U \Longrightarrow x_i\in U\text{ eventually}.$$

Given a topology $\tau$ in $X$, a net $\brackets{x_i}_{i\in I}$ in $X$ is said to \e{converge to $x\in X$ topologically with respect to }, denoted by $(x_i)_{i\in I}\xrightarrow{\tau}x$ or $x_i\xrightarrow{\tau}x$, if for every $U\in tau$ such that $x\in U$, $x_i\in U$ eventually. A net convergence structure $\mathcal{C}$ in $X$ is said to be \emph{topological} if for every net $(x_i)_{i\in I}$ in $X$ which converges topologically to $x$ with respect to $\tau_{\mathcal{C}}$, it holds that $((x_i)_{i\in I},x)\in\mathcal{C}$. The following theorem provides a necessary and sufficient condition for a net convergence structure to be topological.

\begin{theorem}\cite{kel}\label{th:Kelley_char}
A net convergence structure $\mc{C}$ in a set $X$ is topological if and only if it satisfies the following four axioms:
\begin{enumerate}[\hspace{-.5mm}(1)]
\item (Constants).
If $\brackets{x_i}_{i\in I}$ is a constant net with $x_i=x$ for all $i\in I$, then $\brackets{\brackets{x_i}_{i\in I},x}\in \mc{C}$.

\item (Subnets).
If $\brackets{\brackets{x_i}_{i\in I},x} \in \mc{C}$ and $\brackets{y_j}_{j\in J}$ is a subnet of $\brackets{x_i}_{i\in I}$, then $\brackets{\brackets{y_j}_{j\in J},x}\in\mc{C}$.

\item (Divergence).
If $\brackets{\brackets{x_i}_{i\in I},x} \notin \mc{C}$, then there exists a subnet $\brackets{y_j}_{j\in J}$ of $\brackets{x_i}_{i\in I}$ such that for any subnet $\brackets{z_k}_{k\in K}$ of $\brackets{y_j}_{j\in J}$,
$\brackets{\brackets{z_k}_{k\in K},x} \notin \mc{C}$.

\item (Iterated limits).
If $\brackets{\brackets{x_i}_{i\in I},x} \in \mc{C}$ and
$\brackets{\brackets{x_{i,j}}_{j\in J(i)},x_i}\in\mc{C}$ for all $i\in I$, then we have $\brackets{\brackets{x_{(i,f)}}_{(i,f)\in K},x}\in\mc{C}$, where the partial order on
$K :=I\times\prod \himp{J(i)\mid i\in I}$ is the pointwise order.
\end{enumerate}
\end{theorem}

Henceforth convergence structure always refers to net convergence structure. 

\section{Minimal Subset Selections and $\mc{M}$-convergence Structures in Posets}

We begin this section by recalling definition of subset selections.  A \e{subset selection} $\mc{M}$ assigns to each poset $P$ a certain collection $\mc{M}(P)$ of subsets $P$ (see, e.g, \cite{zhao92,zhao15}). In this paper, our concern is on certain subset selections which we call minimal subset selections.

\begin{definition}
For every poset $P$, we define $\mc{M}(P)$ to be a collection of subsets of $P$ such that $\varnothing\notin\mc{M}(X)$ and $\sets{x}\in\mc{M}(P)$ for every $x\in P$. We call the assignment $\mc{M}$ \e{minimal subset selection}. If $A\in\mc{M}(P)$, then we call it an \e{$\mc{M}$-set in $P$}. We denote the collection of all $\mc{M}$-sets having supremum, resp. infimum, by $\mc{M}^+(P)$, resp. $\mc{M}^-(P)$.
\end{definition}

There are many minimal subset selections can be defined. The following are some examples:
\begin{enumerate}
\item $\Dir(P):=$ the collection of all directed subsets of $P$,
\item $\Filt(P):=$ the collection of all filtered subsets of $P$,
\item ${\rm fin}(P):=$ the collection of all nonempty finite subsets of $P$,
\item ${\rm Ch}(P):=$ the collection of all nonempty chains in $P$, and
\item ${\rm ACh}(P):=$ the collection of all nonempty anti-chains in $P$.
\end{enumerate}
It is known that any $T_0$ space $X$ induces a poset, called \e{specialisation poset}. Hence one can define minimal subset selections by making use of some topological property. The following are some of them:
\begin{enumerate}
\item $\Irr(X):=$ the collection of all irreducible sets in $X$,
\item ${\rm Cpt}(X):=$ the collection of all nonempty compact sets in $X$, and
\item ${\rm Con}(P):=$ the collection of all connected sets if $P$.
\end{enumerate}

Throughout this section we assume that a certain minimal subset selection $\mc{M}$ is already given.

\begin{definition} 
Let $P$ be a poset. We say that a net $(x_i)_{i\in I}$ \emph{$\mc{M}$-converges to $x\in P$}, denoted by $(x_i)_{i\in I}\mconv x$ or $x_i\mconv x$, if there exists $A\in\mc{M}^+(P)$ such that $x\leq\bigvee A$ and $A\subseteq\elbb{(x_i)_{i\in I}}$.
\end{definition}

\begin{remark}\label{rem:Mconv_to_down}
Let $P$ be a poset.
\begin{enumerate}
\item[{\rm (1)}] The $\mc{M}$-convergence structure in $P$ satisfies (Constants) and (Subnets) axioms.
\item[{\rm (2)}] If $x_i\mconv x$ implies $x_i\mconv y$ for every $y\in \downarrow x$.
\item[{\rm (3)}] If $\mc{M}=\Dir$ then the $\mc{M}$-convergence structure is exactly the Scott-convergence structure.
\end{enumerate}
\end{remark}

The following proposition provides a way to construct a net out of an $\mc{M}$-set with some special properties.

\begin{proposition}\label{prop:net_from_M_set}
For each $A\in\mc{M}^+(P)$, define $$I_A=\sets{(u,\ub(B))\mid B\sub_{{\rm fin}}A, B\neq\varnothing,\text{ and }u\in\ub(B)}$$ and equip it with the following order: $$(u_1,\ub(B_1))\leq (u_2,\ub(B_2))\t{ if and only if }\ub(B_1)\supseteq \ub(B_2).$$ For each $(u,\ub(B))\in I_A$, define $x_{(u,\ub(B))}=u$ to form a net $(x_i)_{i\in I_A}$. Then $x_i\mconv \bigvee A$.
\end{proposition}

It is known that Scott convergence structures induces Scott topologies and this topology has an order description in its definition. Similarly, we also have an order description for the topology induced by $\mc{M}$-convergence structure.

\begin{proposition}\label{prop:tm_char}
Let $P$ be a poset and $\tau_{\mc{M}}$ the topology induced by the $\mc{M}$-convergence structure in $P$. Then $V\in \tau_{\mc{M}}$ if and only if the following two conditions hold:
\begin{enumerate}[{\rm (TM1)}]
\item $V$ is an upper set;
\item for each $A\in\mc{M}^+(P)$, if $\bigvee A\in V$ then there exists a nonempty finite subset $B$ of $A$ such that $\ub(B)\subseteq V$.
\end{enumerate}
\end{proposition}
\begin{proof}
Let $V\in \tau_{\mc{M}}$. We will prove (TM1) and (TM2) hold.
\begin{enumerate}[{\rm (TM1)}]
\item Let $x\in \uparrow V$. There exists $v\in V$ such that $v\leq x$. By Remark \ref{rem:Mconv_to_down}, we have that the constant net $(x)$ $\mc{M}$-converges to $v$. Hence $x\in V$.
\item Let $A\in\mc{M}^+(P)$ such that $\bigvee A\in V$. We consider the net $(x_i)_{i\in I_A}$ as in Proposition \ref{prop:net_from_M_set}. Since $V\in \tau_{\mc{M}}$, there exists $i_0\in I_A$ such that $j\geq i_0$ implies $x_i\in V$. Let $i_0=(u,\ub(B))$, where $B$ is a nonmepty finite subset of $A$. Then for each $t\in\ub(B)$ it holds that $j:=(t,\ub(B))\geq i_0$, we have that $t\in \ub(B)$. Therefore $\ub(B)\subseteq V$.
\end{enumerate}
Conversely, let $V$ satisfy (TM1) and (TM2), $(x_i)_{i\in I}$ $\mc{M}$-converge to $x$, and $x\in V$. Then there exists $A\in\mc{M}^+(X)$ such that $x\leq \bigvee A$ and $A\subseteq\elb((x_i)_{i\in I})$. By (TM1), $\bigvee A\in V$. By (TM2), there exists a nonempty finite subset $B$ of $A$ such that $\ub(B)\sub V$. Since $B$ is finite and $B\subseteq\elb((x_i)_i\in I)$, we have that $x_i\in\ub(B)$ eventually. Therefore $V\in\tau_{\mc{M}}$.
\end{proof}

We define the \emph{$\mc{M}$-way-below-relation} $\ll_{\mc{M}}$ on a poset $P$ as follows: $x\ll_{\mc{M}} y$ if and only if for every $A\in\mc{M}^+(P)$, if $y\leq\bigvee A$ then there exists a nonempty finite subset $B$ of $A$ such that $\ub(B)\sub\uparrow x$. We then have the following proposition.

\begin{proposition}\label{prop:llm_is_aux}
Let $P$ be a poset and $u,x,y,z\in P$.
\begin{enumerate}[{\rm(1)}]
\item If $x\ll_{\mc{M}}y$ then $x\leq y$.
\item If $u\leq x\ll_{\mc{M}}y\leq z$ then $u\ll_{\mc{M}} z$.
\item If $P$ has a bottom element $\bot$, then $\bot\ll_{\mc{M}}x$ for every $x\in P$.
\item $x\ll_{\mc{M}} y$ if and only if for every net $(x_i)_{i\in I}$, $x_i\mconv  y$ implies $x_i\in\uparrow x$ eventually.
\end{enumerate}
\end{proposition}
\begin{proof}
\begin{enumerate}
\item It is immediate since $\sets{y}\in\mc{M}^+(P)$.
\item Let $A\in\mc{M}^+(P)$ such that $z\leq \bigvee A$. Then $y\leq\bigvee A$, implying the existence of a nonempty finite subset $B$ of $A$ such that $\ub(B)\subseteq \uparrow x\subseteq \uparrow u$.
\item It is clear from the definition of $\ll_{\mc{M}}$.
\item Let $x\ll_{\mc{M}}y$ and $(x_i)_{i\in I}$ be a net $\mc{M}$-converging to $y$. There exists $A\in\mc{M}^+(P)$ such that $y\leq \bigvee A$ and $A\subseteq\elb((x_i)_{i\in I})$. By assumption, there exists a nonempty finite subset $B$ of $A$ such that $\ub(B)\subseteq \uparrow x$. Since $B$ is finite and $B\subseteq\elb((x_i)_{i\in I})$, we have that $x_i\in\ub(B)$ eventually, hence $x_i\in\uparrow x$ eventually. Conversely, assume for every net $(x_i)_{i\in I}$, $x_i\mconv  y$ implies $x_i\in\uparrow x$ eventually and let $A\in\mc{M}^+(P)$ such that $y\leq \bigvee A$. We consider the net $(x_i)_{i\in I_A}$ as given in Proposition \ref{prop:net_from_M_set}. By assumption, there exists $i_0\in I_A$ such that $j\geq i$ implies $x_i\in \uparrow x$. By the construction of $(x_i)_{i\in I_A}$, this implies there exists a nonempty finite subset $B$ of $A$ such that $\ub(B)\subseteq \uparrow x$. We conclude that $x\ll_{\mc{M}}y$.
\end{enumerate}
\end{proof}

Now we are ready to introduce the notion of $\mc{M}$-continuity, which, we will prove later, is a characterisation for $\mc{M}$-convergence structure being topological.

\begin{definition}\label{def:Mcts}
A poset $P$ is said to be \emph{$\mc{M}$-continuous} if for each $x\in P$, the following hold:
\begin{enumerate}[{\rm (M1)}]
\item there exists $A\in\mc{M}^+(X)$ such that $A\subseteq \ddam x:=\sets{y\in P\mid y\llm x}$ and $x\leq\bigvee A$;
\item $\duam x:=\sets{y\in P\mid x\llm y}\in\tau_{\mc{M}}$.
\end{enumerate}
\end{definition}

Note that, in virtue of Remark \ref{rem:Mconv_to_down} and Proposition \ref{prop:tm_char}, to prove (M2), it suffices to show that $\duam x$ satisfies (TM2).

\begin{remark}
Let $P$ be a poset.
\begin{enumerate}
\item[{\rm (1)}] If $\mc{M}(P)$ is the collection of all directed subsets of $P$ then $\llm$ is just the usual way-below relation on $P$ and $P$ is $\mc{M}$-continuous if and only if $P$ is a continuous poset. In particular, (M2) can be implied from (M1).
\item[{\rm (2)}] If $P$ is $\alpha(\mc{M})$-continuous in the sense of \cite[Definition 3.4]{zhouzhao07}, then for every $x\in P$, (M1) holds.
\end{enumerate}
\end{remark}

\begin{lemma}\label{lem:Mctsimplies_top}
If $P$ is an $\mc{M}$-continuous poset, then the $\mc{M}$-convergence structure in $P$ is topological.
\end{lemma}
\begin{proof}
It suffices to show that for each net $(x_i)_{i\in I}$, $(x_i)_{i\in I}\xrightarrow{\tau_{\mc{M}}}x$ implies $(x_i)_{i\in I}\mconv x$. Assume that $(x_i)_{i\in I}\xrightarrow{\tau_{\mc{M}}}x$. By (M2), there exists $A\in\mc{M}^+(X)$ such that $A\sub\ddam x$ and $x\leq\bigvee A$. Let $a\in A$. Since $x\in\duam a\in\tau_{\mc{M}}$, we have that $x_i\in \duam a$ eventually, implying $x_i\geq a$ eventually. Hence $(x_i)_{i\in I}\mconv x$, as desired.
\end{proof}

The following lemma informs us that $\mc{M}$-convergence structure in a poset satisfying only (Iterated limits) axiom is enough to deduce the poset is $\mc{M}$-continuous. 

\begin{lemma}\label{lem:iter_implies_Mcts}
Let $P$ be a poset. If the $\mc{M}$-convergence structure in $P$ satisfies (Iterated limits) axiom then $P$ is $\mc{M}$-continuous.
\end{lemma}
\begin{proof}
Let $x$ be an arbitrary element of $P$. We will show that (M1) and (M2) hold.
\begin{enumerate}[{\rm (M1)}]
\item Define the following collection
$$\mc{A}_x=\sets{A\in\mc{M}^+(P)\mid x\leq\bigvee A}.$$
Let $\mc{A}_x=\sets{A_i\mid i\in I}$. Equip $I$ with $\leq$ defined as $i\leq j$ for each $i,j\in I$. Then $I$ is a directed pre-ordered set. For each $i\in I$, let $x_i=\bigvee A_i$ to form a net $(x_i)_{i\in I}$. Since $\sets{x}\in\mc{M}^+(P)$, we then have $x_i\mconv x$. For each $i\in I$, we have a net $(x_{i,j})_{j\in I_{A_i}}$ $\mc{M}$-converging to $x_i$ as in Proposition \ref{prop:net_from_M_set}. By assumption we have that
$$(x_{(i,f)})_{(i,f)\in K}\mconv x,$$
where $K:=I\times \Pi_{i\in I}I_{A_i}$ is ordered by pointwise order and $x_{(i,f)}=x_{i,f(i)}$ for every $(i,f)\in K$. Then there exists $A\in \mc{M}^+(P)$ such that $x\leq \bigvee A$ and $A\subseteq\elb((x_{(i,f)})_{(i,f)\in K})$. It remains to show that $A\subseteq\ddam x$. Let $a\in A$ and $M\in\mc{M}^+(P)$ such that $x\leq \bigvee M$. Then there exists $i_0\in I$ and $(i^*,f^*)\in K$ such that $M=A_{i_0}$ and for every $(i,f)\in K$, $(i,f)\geq (i^*,f^*)$ implies $x_{i,f(i)}\geq a$. Let $f^*(i_0)=(u,\ub(B))\in I_{A_{i_0}}$, for some nonempty finite subset $B$ of $A_{i_0}$, and $t\in \ub(B)$. Define $f\in\Pi_{i\in I}I_{A_i}$ as follows
$$
f(i)=\left\{
\begin{array}{ll}
f^*(i),&\text{if }i\neq i_0\\
(t,\ub(B)), &\text{if }i=i_0
\end{array}
\right.
$$
We have that $(i_0,f)\geq (i^*,f^*)$. Hence $x_{i_0,f(i_0)}=x_{i_0,(t,\ub(B))}=t\geq a$. Therefore $\ub(B)\sub\uparrow a$ and hence $a\llm x$.
 
\item Let $A\in\mc{M}^+(X)$ such that $\bigvee A\in \duam x$. Suppose to the contrary that for each nonempty finite subset $i$ of $A$ there exists $x_i\in \ub(i)$ such that $x_i\notin\duam x$. Let $I$ be the collection of all nonempty finite subsets of $A$. Equipping $I$ with inclusion order, we have that $I$ is a directed pre-ordered set. Hence $(x_i)_{i\in I}$ is a net. If $a\in A$, then for every $i\in I$ such that $i\geq \sets{a}$ we have that $x_i\in \ub(i)\sub \uparrow a$. We have that $x_i\mconv \bigvee A$.
Let $i\in I$. Since $x\not\llm x_i$, in virtue of Proposition \ref{prop:llm_is_aux}(4), there exists a net $(x_{i,j})_{j\in J(i)}$ such that $x_{i,j}\mconv x_i$ and
\begin{equation*}\label{cont}
\forall j\in J(i).~\exists j^*\in J(i).~j^*\geq j \t{ and }x_{i,j^*}\ngeq x \qquad \text{(St1)}
\end{equation*}
By the fact that $\mc{M}$-convergence structure in $P$ satisfies (Iterated limits) axiom, we have that
$$(x_{(i,f)})_{(i,f)\in K}\mconv \bigvee A,$$
where $K:=I\times \Pi_{i\in I}J(i)$ is ordered by pointwise order and $x_{(i,f)}=x_{i,f(i)}$ for every $(i,f)\in K$.
Since $x\llm \bigvee A$, by Proposition \ref{prop:llm_is_aux}(4), there exists $(i_0,f_0)\in K$ such that for every $(i,f)\in K$, $(i,f)\geq (i_0,f_0)$ implies $x_{i,f(i)}\geq x$. Let $f_0(i_0)=j\in J(i_0)$. Let $j^*\in J(i_0)$ satisfy (St1). Define $f\in\Pi_{i\in I}J(i)$ as follows
$$
f(i)=\left\{
\begin{array}{ll}
f_0(i),&\text{if }i\neq i_0\\
j^*, &\text{if }i=i_0
\end{array}
\right.
$$
We have that $(i_0,f)\geq (i_0,f_0)$. Hence $x_{i_0,f(i_0)}=x_{i_0,j*}\geq x$, which contradicts (St1). Therefore we have that there exists a nonempty finite subset $B$ of $A$ such that $\ub(B)\sub \duam x$. Thus $\duam x\in\tau_{\mc{M}}$.
\end{enumerate}
\end{proof}

The following theorem is an immediate consequence of Lemmas \ref{lem:Mctsimplies_top} and \ref{lem:iter_implies_Mcts}.
\begin{theorem}\label{th:Mcts_equiv}
The following statements are equivalent for a poset $P$.
\begin{enumerate}
\item[{\rm (1)}] $P$ is $\mc{M}$-continuous.
\item[{\rm (2)}] The $\mc{M}$-convergence structure in $P$ is topological.
\end{enumerate}
\end{theorem}

As a corollary of Theorem \ref{th:Mcts_equiv} above, we have the following well-known result.
\begin{corollary}\cite{zhaozhao05}
Let $P$ be a poset. Then the Scott-convergence structure in $P$ is topological if and only $P$ is a continuous poset.
\end{corollary}

Notice that $\mc{M}$-convergence structures can be defined in any poset. In particular, given a $T_0$ space $X$, we have the $\mc{M}$-convergence structure in the specialisation poset induced by $X$. Consider the assignment $\Irr$ assigning any space to the collection $\Irr(X)$ of all irreducible sets in $X$. Clearly $\Irr(X)$ contains all singletons. Hence $\Irr$ can be considered as a minimal subset selection. Making use of this subset selection, one may define $\Irr$-convergence structure (see, e.g., \cite{andradiho17,zhaolu17}).

\begin{definition}
A $T_0$ space is said to be \e{$\Irr$-continuous} if its induced specialisation poset is $\Irr$-continuous in the sense of Definition \ref{def:Mcts}, when $\Irr$ is considered as a minimal subset selection.
\end{definition}

The following result is a direct consequence of Theorem \ref{th:Mcts_equiv}.

\begin{corollary}\cite{zhaolu17}
Let $X$ be a space. Then the $\Irr$-convergence structure in $X$ is topological if and only if $X$ is $\Irr$-continuous.
\end{corollary}

We recall the following definition and theorem from \cite{zhouzhao07}.
\begin{definition}\cite{zhouzhao07}
A poset $P$ is called $\alpha(\mc{M})$-continuous if for every $x\in P$, $\ddam x$ is an $\mc{M}$-set whose supremum equals $x$.
\end{definition}
\begin{theorem}\cite[Theorem 3.1]{zhouzhao07}\label{th:fromzz07}
Let $P$ be a poset such that the following two conditions hold for every $x\in P$:
\begin{enumerate}
\item[{\rm (1)}] $\ddam x\in\mc{M}(P)$, and
\item[{\rm (2)}] $\sets{y\in P\mid \exists z\in P\t{ s.t. }y\llm z\llm x}\in \mc{M}(P)$.
\end{enumerate}
Then the $\mc{M}$-convergence structure in $P$ is topological if and only if $P$ is $\alpha(\mc{M})$-continuous.
\end{theorem}

One can see that Theorem \ref{th:fromzz07} is not complete in the sense of \cite{sunli18}. More precisely, the equivalence given in the theorem only holds in a certain class of posets. In the following we provide a minimal subset selection $\mc{M}$ and a poset which is not in the class but the $\mc{M}$-convergence structure in it is topological.

\begin{example}
Let $\mc{M}={\rm ACh}$ and $P=\sets{a,b,c}$ with $a\leq c$ and $b\leq c$. We have that $\mc{M}(P)=\mc{M}^+(P)=\sets{\sets{a},\sets{b},\sets{c},\sets{a,b}}$. It is easy to verify that $x\llm y$ if and only if $x\leq y$. We then have that $P$ is $\mc{M}$-continuous. According to Theorem \ref{th:Mcts_equiv}, the $\mc{M}$-convergence structure in $P$ is topological. But we have that $\ddam c=P$ is not an antichain, which implies that $P$ is not $\alpha(\mc{M})$-continuous.
\end{example}

\section{$\mc{MN}$-convergence Structures in Posets}

In this section, we move our focus on a certain generalisation of order-convergence structure. Throughout this section we assume that two certain minimal subset selections $\mc{M}$ and $\mc{N}$ are already given.

\begin{definition}\label{def:MNconve}
Let $P$ be a poset. We say that a net $(x_i)_{i\in I}$ \emph{$\mc{MN}$-converges to $x\in P$}, denoted by $(x_i)_{i\in I}\mnconv x$ or $x_i\mnconv x$, if there exist $M\in\mc{M}^+(P)$ and $S\in\mc{N}^-(P)$ such that
\begin{enumerate}
\item $x=\bigvee A=\bigwedge S$, and
\item for each $a\in A$ and $s\in S$, $x_i\in\uparrow a\cap \downarrow s$ eventually.
\end{enumerate}
\end{definition}

\begin{remark}
\begin{enumerate}[{\rm (1)}]
\item A net can only $\mc{MN}$-converge to at most one element.
\item The $\mc{MN}$-convergence structure in a poset $P$ satisfies (Constants) and (Subnets) axioms.
\end{enumerate}
\end{remark}

Similar to $\mc{M}$-convergence case, we can construct a certain net from given $\mc{M}$-set and $\mc{N}$-set, and give an order description of the topology induced by $\mc{MN}$-convergence structure.

\begin{proposition}\label{prop:net_from_MN_set}
For each $A\in\mc{M}^+(P)$ and $S\in\mc{N}^-(P)$ such that $\bigvee A=\bigwedge S=:x$, define $$I_{AS}=\sets{(u,\ub(B)\cap\lb(T))\mid B\subseteq_{{\rm fin}} A,T\subseteq_{{\rm fin}} S, B\neq\varnothing,T\neq\varnothing,\t{ and }u\in\ub(B)\cap\lb(T)}$$ and equip it with the following order:
$$(u_1,\ub(B_1)\cap\lb(T_1))\leq (u_2,\ub(B_2)\cap\lb(T_2))\t{ if and only if }\ub(B_1)\cap\lb(T_1)\supseteq \ub(B_2)\cap\lb(T_2).$$ For each $(u,\ub(B)\cap\lb(T))\in I_{AS}$, define $x_{(u,\ub(B)\cap\lb(T))}=u$ to form a net $(x_i)_{i\in I_{AS}}$. Then $x_i\mnconv x$.
\end{proposition}

\begin{proposition}\label{prop:tmn_char}
Let $P$ be a poset and $\tau_{\mc{MN}}$ be the topology induced by the $\mc{MN}$-convergence structure in $P$. The following two statements are equivalent.
\begin{enumerate}
\item[{\rm (1)}] $V\in \tau_{\mc{MN}}$.
\item[{\rm (2)}] For every $A\in\mc{M}^+(P)$ and $S\in\mc{N}^-(P)$, if $\bigvee A=\bigwedge S\in V$ then there exist nonempty finite subsets $B$ of $A$ and $T$ of $S$ such that $\ub(B)\cap\lb(T)\subseteq V$.
\end{enumerate}
\end{proposition}
\begin{proof}
Let $V\in \tau_{\mc{M}}$, $A\in\mc{M}^+(P)$, and $S\in\mc{N}^-(P)$ such that $\bigvee A=\bigwedge S\in V$. We consider the net $(x_i)_{i\in I_{AS}}$ as in Proposition \ref{prop:net_from_MN_set}. Since $V\in \tau_{\mc{MN}}$, there exists $i_0\in I_{AS}$ such that $j\geq i_0$ implies $x_i\in V$. Let $i_0=(u,\ub(B)\cap\lb(T))$. Then for each $v\in\ub(B)\cap\lb(T)$ it holds that $j:=(v,\ub(B)\cap\lb(T))\geq i_0$, we have that $v\in V$. Therefore $\ub(B)\cap\lb(T)\subseteq V$.

Conversely, let $(x_i)_{i\in I}$ $\mc{MN}$-converge to $x$, and $x\in V$. Then there exists $A\in\mc{M}^+(P)$ and $S\in\mc{N}^-(P)$ such that $x= \bigvee A=\bigwedge S$ and for each $a\in A$ and $s\in S$, $x_i\in\uparrow a\cap\downarrow s$ eventually. By assumption, there exist nonempty finite subsets $B$ of $A$ and $T$ of $S$ such that $\ub(B)\cap\lb(T)\sub V$. Since $B$ and $T$ are finite and for each $a\in A$ and $s\in S$, $x_i\in\uparrow a\cap\downarrow s$ eventually, we have that $x_i\in\ub(B)\cap\lb(T)$ eventually. Therefore $V\in\tau_{\mc{MN}}$.
\end{proof}

Next we define two following relations on a poset making use of $\mc{M}$-set and $\mc{N}$-set:
\begin{enumerate}
\item[{\rm (1)}] $x\llmn y$ if and only if for every $A\in\mc{M}^+(P)$ and $S\in\mc{N}^-(P)$ such that $y=\bigvee A=\bigwedge S$ there exist nonempty finite subsets $B$ of $A$ and $T$ of $S$ such that $\ub(B)\cap\lb(T)\sub\uparrow x$;
\item[{\rm (2)}] $y\triangleleft_{\mc{MN}} z$ if and only if for every $A\in\mc{M}^+(P)$ and $S\in\mc{N}^-(P)$ such that $y=\bigvee A=\bigwedge S$ there exist nonempty finite subsets $B$ of $A$ and $T$ of $S$ such that $\ub(B)\cap\lb(T)\sub\downarrow z$.
\end{enumerate}

\begin{proposition}\label{prop:llmn_and_trmn}
For a poset $P$ and $x,y,z\in P$ the following hold.
\begin{enumerate}[{\rm(1)}]
\item $x\llmn y$ implies $x\leq y$.
\item $y\trmn z$ implies $y\leq z$.
\item If $P$ has a bottom element $\bot$, then $\bot\llmn x$ for every $x\in P$.
\item If $P$ has a top element $\top$, then $x\trmn \top$ for every $x\in P$.
\item $x\llmn y$ if and only if for every net $\brackets{x_i}_{i\in I}$, $x_i\mnconv y$ implies $x_i\in\uparrow x$ eventually.
\item $y\trmn z$ if and only if for every net $\brackets{x_i}_{i\in I}$, $x_i\mnconv y$ implies $x_i\in\downarrow z$ eventually.
\item $x\llmn y\trmn z$ if and only if for every net $\brackets{x_i}_{i\in I}$, $x_i\mnconv y$ implies $x_i\in\uparrow x\cap \downarrow z$ eventually.
\item If $x\llmn y$, then for every $A\in\mc{M}^+(P)$, $\bigvee A=y$ implies there exists a nonempty finite subset $B$ of $A$ such that $\ub(B)\cap\lb(\sets{y})\sub\uparrow x$.
\item If $y\trmn z$, then for every $S\in\mc{N}^-(P)$, there exists a nonempty finite subset $T$ of $S$ such that $\ub(\sets{y})\cap\lb(T)\sub \downarrow z$.
\end{enumerate}
\end{proposition}
\begin{proof}
(1), (2), (8), and (9) are immediate since $\sets{y}\in\mc{M^+(P)}\cap\mc{N}^-(P)$, (3) and (4) are clear from the definition of $\llmn$ and $\trmn$, and (7) is correct once (5) and (6) are.
\begin{enumerate}[{\rm (1)}]
\item[{\rm (5)}] Let $x\llmn y$ and $(x_i)_{i\in I}$ be a net $\mc{MN}$-converging to $y$. There exist $A\in\mc{M}^+(P)$ and $S\in\mc{N}^-(P)$ such that $y= \bigvee A=\bigwedge S$ and for every $a\in A$ and $s\in S$, $x_i\in \uparrow a\cap\downarrow s$ eventually. By assumption, there exist nonempty finite subsets $B$ of $A$ and $T$ of $S$ such that $\ub(B)\cap\lb(T)\sub \uparrow x$. Since $B$ and $T$ are finite, we have that $x_i\in\ub(B)\cap\lb(T)$ eventually, hence $x_i\in\uparrow x$ eventually. Conversely, assume for every net $(x_i)_{i\in I}$, $x_i\mnconv y$ implies $x_i\in\uparrow x$ eventually and let $A\in\mc{M}^+(P)$ and $S\in\mc{N}^-(P)$ such that $y= \bigvee A=\bigwedge S$. We consider the net $(x_i)_{i\in I_{AS}}$ as given in Proposition \ref{prop:net_from_MN_set}. By assumption, there exists $i_0\in I_{AS}$ such that $j\geq i$ implies $x_i\in \uparrow x$. By the construction of $(x_i)_{i\in I_{AS}}$, this implies there exist nonempty finite subsets $B$ of $A$ and $T$ of $S$ such that $\ub(B)\cap\lb(T)\subseteq \uparrow x$. We conclude that $x\llmn y$.

\item[{\rm (6)}] The proof is similar to (5).
\end{enumerate}
\end{proof}

\begin{remark}
If $\mc{M}=\Dir$ and $\mc{N}=\Filt$, then
\begin{enumerate}[{\rm (1)}]
\item $\mc{MN}$-convergence structure is exactly order-convergence structure,
\item $\llmn$ is precisely $\ll_{\mc{O}}$ given in \cite{sunli18} and \cite{wangzhao13}.
\end{enumerate}
\end{remark}

The following denotations will be needed in our ensuing development.
\begin{enumerate}
\item $ \ddamn x=\sets{y\in P\mid y\llmn x}$,
\item $ \duamn x=\sets{y\in P\mid x\llmn y}$,
\item $ \ddatrmn x=\sets{y\in P\mid y\trmn x}$,
\item $ \duatrmn x=\sets{y\in P\mid x\trmn y}$.
\end{enumerate}

In the following, we present our new notion of $\mc{MN}$-continuity of posets. It will be proven that $\mc{MN}$-continuity of a poset is equivalent with the $\mc{MN}$-convergence in it being topological.

\begin{definition}\label{def:MNcts}
A poset $P$ is said to be \emph{$\mc{MN}$-continuous} if for each $x,y\in P$, the following hold:
\begin{enumerate}[{\rm (MN1)}]
\item there exist $A\in\mc{M}^+(X)$ and $S\in\mc{N}^-(X)$ such that $A\subseteq \ddamn x$, $S\sub \duatrmn x$, and $x=\bigvee A=\bigwedge S$;
\item $\duamn x\cap \ddatrmn y\in\tau_{\mc{MN}}$.
\end{enumerate}
\end{definition}

\begin{lemma}\label{lem:MNctsimplies_top}
If $P$ is an $\mc{MN}$-continuous poset, then the $\mc{MN}$-convergence structure in $P$ is topological.
\end{lemma}
\begin{proof}
It suffices to show that for each net $(x_i)_{i\in I}$ in $P$, $(x_i)_{i\in I}\xrightarrow{\tau_{\mc{MN}}}x$ implies $(x_i)_{i\in I}\mnconv x$. Assume that $(x_i)_{i\in I}\xrightarrow{\tau_{\mc{MN}}}x$. By (MN1), there exist $A\in\mc{M}^+(X)$ and $S\in\mc{N}^-(P)$ such that $A\sub\ddamn x$, $S\sub \duatrmn x$, and $x=\bigvee A=\bigwedge S$. Let $a\in A$ and $s\in S$. Since $x\in\duamn a\cap \ddatrmn s$ and $\duamn a\cap \ddatrmn s\in\tau_{\mc{MN}}$ by (MN2), we have that $x_i\in \duamn a\cap \ddatrmn s$ eventually. Hence, by Proposition \ref{prop:llmn_and_trmn}(1) and (2), $x_i\in\uparrow a\cap\downarrow s$ eventually. We conclude that $(x_i)_{i\in I}\mnconv x$.
\end{proof}

\begin{lemma}\label{lem:iter_implies_MNcts}
Let $P$ be a poset. If the $\mc{MN}$-convergence structure in $P$ is topological then $P$ is $\mc{MN}$-continuous.
\end{lemma}
\begin{proof}
Let $x$ and $y$ be arbitrary elements of $P$. We will show that (MN1) and (MN2) hold.
\begin{enumerate}[{\rm (MN1)}]
\item Define the following collection
$$I_x=\sets{(V,c)\in \tau_{\mc{MN}}\times P\mid
x,c\in V}$$
Equip $I_x$ with $\leq$ defined as $(V_1,c_1)\leq (V_2,c_2)$ if and only if $V_1\supseteq V_2$. Then $I_x$ is a directed pre-ordered set. For each $(V,c)\in I_x$, we define $x_{(V,c)}=c$ to form a net. For each $V\in\tau_{\mc{MN}}$ such that $x\in V$ we have that $(V,x)\in I_x$. For each $(W,c)\in I_x$ such that $(W,c)\geq (V,x)$ we have $x_{(W,c)}=c\in W\sub V$, hence $x_i\in V$ eventually. Therefore $x_i\xrightarrow{\tau_{\mc{MN}}}x$. Since the $\mc{MN}$-convergence in $P$ is topological, we have that $x_i\mnconv x$. Then there exists $A\in\mc{M}^+(P)$ and $S\in\mc{N}^-(X)$ such that $x=\bigvee A=\bigwedge S$ such that for every $a\in A$ and $s\in S$, $x_i\in\uparrow a\cap \uparrow s$ eventually.

We fixed $a\in A$ and $s\in S$. There exists $(V,c)\in I_x$ such that $(W,d)\geq (V,c)$ implies $x_{(W,d)}\in\uparrow a\cap \downarrow s$. Since for every $v\in V$ we have $(V,v)\geq (V,c)$, it holds that $V\sub \uparrow a\cap\downarrow s$. Now let $\brackets{y_j}_{j\in J}$ be a net such that $y_j\mnconv x$. Then $y_j\xrightarrow{\tau_{\mc{MN}}}x$. Since $x\in V\in \tau_{\mc{MN}}$, we have that $y_j\in V\sub \uparrow a\cap\downarrow s$ eventually. By Proposition \ref{prop:llmn_and_trmn}(7), we have that $a\llmn x\trmn s$. Since $a$ and $s$ are taken arbitrarily, we have that $A\sub\duamn x$ and $S\sub\duatrmn x$.
 
\item Suppose to the contrary that $\duamn x\cap \ddatrmn y\notin\tau_{\mc{MN}}$. There exists $z\in \duamn x\cap \ddatrmn y$ such that for every $V\in I:=\sets{V\in\tau_{\mc{MN}}\mid z\in V}$ there exists $x_V\in V$ such that $x_V\notin \duamn x\cap \ddatrmn y$. Equipping $I$ with reverse inclusion order we have that the net $\brackets{x_V}_{V\in I}$ converges to $z$ with respect to $\tau_{\mc{MN}}$. Since the $\mc{MN}$-convergence structure in $P$ is topological, we have that $x_V\mnconv z$.

For each $V\in I$, since $x_V\notin \duamn x\cap\ddatrmn y$, by Proposition \ref{prop:llmn_and_trmn}(7) there exists a net $\brackets{x_{V,j}}_{j\in J(V)}$ such that $x_{j,V}\mnconv x_V$ and 
\begin{equation*}\label{contr}
\forall j\in J(V).~\exists j^*\in J(V).~j^*\geq j\t{ and }x_{V,j^*}\notin\uparrow x\cap\downarrow y\qquad{\rm (St2)}
\end{equation*}
By the fact that $\mc{M}$-convergence structure in $P$ satisfies (Iterated limits) axiom, we have that
$$(x_{(V,f)})_{(V,f)\in K}\mnconv z$$
where $K:=I\times \Pi_{V\in I}J(V)$ is ordered by pointwise order and $x_{(V,f)}=x_{V,f(V)}$ for every $(V,f)\in K$.
Since $z\in\duamn x\cap\ddatrmn y$, by Proposition \ref{prop:llmn_and_trmn}(7), there exists $(V_0,f_0)\in K$ such that for every $(V,f)\in K$, $(V,f)\geq (V_0,f_0)$ implies $x\leq x_{V,f(V)}\leq y$. Let $f_0(V_0)=j\in J(V_0)$. Let $j^*\in J(V_0)$ satisfy (St2). Define $f\in\Pi_{V\in I}J(V)$ as follows
$$
f(V)=\left\{
\begin{array}{ll}
f_0(V),&\text{if }V\neq V_0\\
j^*, &\text{if }V=V_0
\end{array}
\right.
$$
We have that $(V_0,f)\geq (V_0,f_0)$. Hence $x\leq x_{V_0,f(V_0)}=x_{V_0,j*}\leq y$, which contradicts (St2). Therefore we have that $\duamn x\cap\ddatrmn y\in\tau_{\mc{MN}}$.
\end{enumerate}
\end{proof}

The following theorem is an immediate consequence of Lemmas \ref{lem:MNctsimplies_top} and \ref{lem:iter_implies_MNcts}.
\begin{theorem}\label{th:MNcts_equiv}
The following statements are equivalent for a poset $P$.
\begin{enumerate}
\item[{\rm (1)}] $P$ is $\mc{MN}$-continuous.
\item[{\rm (2)}] The $\mc{MN}$-convergence structure in $P$ is topological.
\end{enumerate}
\end{theorem}

Notice that considering $\Dir$ and $\Filt$ as minimal subset selections, we have that the $(\Dir\Filt)$-convergence structures in $P$ is exactly the order-convergence structure in $P$. Hence as a corollary of Theorem \ref{th:MNcts_equiv} above, we have the following result.
\begin{corollary}\label{cor:order-conv_is_top}
Let $P$ be a poset. Then the order-convergence structure in $P$ is topological if and only $P$ is a $(\Dir\Filt)$-continuous poset.
\end{corollary}

We recall the following definition from \cite{sunli18}.

\begin{definition}
Let $P$ be a poset and $\mc{O}$ be the order-convergence structure in $P$. We call $P$ is \e{$\mc{R}^*$-doubly continuous} if for every $x,y,z\in P$ the following hold:
\begin{enumerate}[{\rm (R1)}]
\item $\dda_{\mc{O}}x\in\Dir^+(P)$, $\duatr_{\mc{O}}x\in\Filt^-(P)$, and $x=\bigvee \dda_{\mc{O}}x=\bigwedge \duatr_{\mc{O}}x$;
\item $x\ll_{\mc{O}}y\triangleleft_{\mc{O}}z$ implies there exist $a,s\in P$ such that $a\ll_{\mc{O}} x\triangleleft_{\mc{O}}s$ and $\uparrow a\cap\downarrow s\sub \dua_{\mc{O}}y\cap \ddatr_{\mc{O}}z$.
\end{enumerate}
\end{definition}

As expected, specialising $\mc{M}$ and $\mc{N}$ to $\Dir$ and $\Filt$, we have that $\mc{MN}$-continuity is exactly $\mc{R}^*$-doubly continuity.

\begin{proposition}\label{prop:Rcts_equiv_DirFiltcts}
A poset $P$ is $\mc{R}^*$-doubly continuous if and only if $P$ is $(\Dir\Filt)$-continuous. 
\end{proposition}

As a consequence of Corollary \ref{cor:order-conv_is_top} and Proposition \ref{prop:Rcts_equiv_DirFiltcts} we have the following previously-established result.
\begin{corollary}\cite{sunli18}\label{th:order-conv_is_top}
The order-convergence structure in a poset $P$ is topological if and only if $P$ is $\mc{R}^*$-doubly continuous.
\end{corollary}

\end{document}